\documentclass{article}
\usepackage{graphicx} 
\usepackage{amsmath,amsfonts,amssymb}
\usepackage{amsthm}
\usepackage{mathtools}
\usepackage{mathrsfs}
\usepackage{hyperref}
\usepackage[mathscr]{euscript}
\usepackage{enumerate}
\usepackage{authblk}
\newtheorem{theorem}{Theorem}[section]
\newtheorem{theorema}{Theorem}

\newtheorem{corollary}[theorem]{Corollary}
\newtheorem{lemma}[theorem]{Lemma}
\theoremstyle{definition}
\newtheorem{definition}{Definition}[section]
\newtheorem{example}{Example}
\newtheorem*{counterexample}{Counterexample}
\theoremstyle{remark}
\newtheorem{remark}{Remark}
\author[1]{Sonam Mehta}
\author[2]{Kuldeep Singh Charak}

\affil[1]{Department of Mathematics, University of Jammu \\ \href{mailto:mehtasonam2020@gmail.com}{mehtasonam2020@gmail.com}}
\affil[2]{Department of Mathematics, University of Jammu \\ \href{mailto:kscharak7@rediffmail.com}{kscharak7@rediffmail.com}}
\title{Normality criterion for a family of holomorphic curves that partially share wandering hyperplanes with their derivatives, and holomorphic functions lifted to curves in $\PtwoC$}
\date{}
\newcommand{\C}{\mathbb{C}}
\newcommand{\Pro}{\mathbb{P}}
\newcommand{\PnC}{\Pro^N(\C)}
\newcommand{\PtwoC}{\Pro^2(\C)}
\newcommand{\scF}{\mathcal{F}}
\newcommand{\D}{\mathcal{D}}
\begin{document}
\maketitle
\begin{abstract}
    In this paper we  generalize a result  of Ye, Pang and Yang \cite{ye2015extension} on the normality of a family of holomorphic curves in $\Pro^N(\C)$. Further we obtain a normality criteria for family of meromorphic functions that partially share wandering holomorphic functions with their derivatives. We also devise a tractable representation of complex valued holomorphic functions on $D$ as functions from $D$ to $\Pro ^2(\C)$ to obtain a normality criterion that leads to a counterexample to the converse of Bloch's principle.
\end{abstract}
\section{Introduction}
 A family $\scF$ of holomorphic functions on a domain $D$ in the complex plane is said to be normal if any sequence $\{f_n\} \subset \scF$ has a subsequence  which converges locally uniformly in $D$ to an analytic function $f$ or a subsequence which converges locally uniformly in $D$ to $\infty$. Normality for families of meromorphic function is defined likewise with convergence of the subsequence required in the spherical metric of the Riemann sphere.

 A central result in the theory of normal families is the \textit{Crit\'ere fondamental} proved by Montel  in 1912. 
 \begin{theorema}[\cite{montel1912familles}]\label{FNT}
     Let $\scF$ be a family of meromorphic functions on a domain $D\subseteq \mathbb{C}$ which omit three distinct values $a.b,c$ in $\hat{\C}$. Then $\scF$ is normal in $D$.
 \end{theorema}
 
 Recall that two nonconstant meromorphic functions $f$ and $g$ on a domain $D$ \textit{share} the value $a \in \hat{\C}$ if $\{f^{-1}(a)\} = \{g^{-1}(a)\}$ 
 (ignoring multiplicity). \mbox{W. Schwick} in 1992, drew a connection between  normality and shared values with the following theorem:
 \begin{theorema}[\cite{schwick1992sharing}]\label{Schwick theorem}
    Let $\scF$ be a family of meromorphic functions on a domain $D\subseteq \mathbb{C}$, and $a,b,c$ be three distinct finite complex numbers. If for every $f \in \scF$ , $f$ and $f'$ share the values $a,b,c$ , then $\scF$ is normal in $D$. 
 \end{theorema}
 Soon enough the  Theorem \ref{FNT} was generalized to the case of sharing of values by Sun \cite{sun1994shared} where instead of omitting the three values it is only required that the three distinct values be shared by all the functions of the family.

 Theorem \ref{FNT} has also been extended to the family of holomorphic maps into complex projective space as follows:
 \begin{theorema}[\cite{dufresnoy1944theorie}]\label{dufresnoy theorem}
  Let $\scF$ be a family of holomorphic maps of a domain $D$ into  $\Pro^N(\C)$. Let $H_1,H_2, ......, H_{2N+1}$ be hyperplanes in $\Pro^N(\C)$  in general position. If for each $f \in \scF$, $f$ omits  $H_1,H_2, ......, H_{2N+1}$, then $\scF$ is normal in $D$.
 \end{theorema}
 The requirement of omission $2N + 1$ hyperplanes in Theorem \ref{dufresnoy theorem}
 was reduced to merely sharing of such hyperplanes by Yang et al \cite{yang2014normal}. An adequate extension of the theorem  \ref{Schwick theorem} to the case of holomorphic maps into complex projective space could only be achieved after a proper extension of the notion of derivatives for such maps which was accomplished by Ye et al \cite{ye2015extension}.

Let $f$ and $g$ be  two meromorphic functions on a domain $D \subseteq \C$. Then $f$ and $g$ are said to share a function $\alpha$, in $D$, ignoring multiplicities, if $f(z) - \alpha(z) = 0 \iff g(z) - \alpha(z) = 0$. If  $f(z) - \alpha(z) = 0 \implies g(z) - \alpha(z) = 0$ then we say that $f$ and $g$ partially share $\alpha$ in $D$.  

 In 2014 Grahl and Nevo, using a   simultaneous rescaling version of the Zalcman Lemma, obtained a normality criterion for a family of meromorphic functions whose each member shared three functions  with its derivative. 
\begin{theorema}\label{Grahl and Nevo}[\cite{grahl2014exceptional}]
 Let $\scF$ be a family of functions meromorphic in the unit disk  $\mathbb{D} = \{z \in \C: |z| < 1\}$ and $\epsilon >0$. Assume that for each $f \in \scF$ there exist functions $a_1^{(f)}, a_2^{(f)}, a_3^{(f)}$ that are either meromorphic or identically infinity in $\mathbb{D}$, and do not have any common poles with $f$, such that $f$ and $f'$ share the functions $a_1^{(f)}, a_2^{(f)}, a_3^{(f)}$, and such that
 $$\sigma(a_j^{(f)(z)}, a_k^{(f)(z)}) \geq \epsilon \ \forall z \in \mathbb{D}, \ \text{j,k} \ \in \{1,2,3\} \ \text{with} \ j \neq k$$
 Then $\scF$ is normal in $\mathbb{D}.$
\end{theorema}
The functions $a_1^{(f)}, a_2^{(f)}, a_3^{(f)}$ are called \textit{wandering shared functions} as they vary for different functions of the family $\scF.$ After proving our first main result we shall be in a position to replace sharing in Theorem \ref{Grahl and Nevo} by partial sharing. But it shall entail some extra restrictions on the wandering shared functions; that they must be bounded holomorphic functions on $\mathbb{D}.$
  
 The extension of results in theory of meromorphic functions to that of holomorphic functions into $\Pro^N(\C)$ has been a regular feature in the field of normal families. The results obtained by various authors  are more general and can be reduced to the case of meromorphic functions by just taking $N=1$. In our approach to the last result of this paper we have followed a  different path by considering a family of complex valued holomorphic functions and lifting them to  functions into $\Pro^2(\C)$. A choice of proper hyperplanes has enabled us to formulate a nice criteria for normality of complex valued holomorphic functions on $D$. Now we introduce the necessary preliminaries.

\subsection{Projective spaces and hyperplanes}
An $N-$ dimensional complex projective space is $\Pro^N(\C) = \left(\C^{N + 1} \setminus \{0\}\right) / \sim $, where  $(a_0,...,a_N) \sim (b_0,...,b_N)$ if and only if $(a_0,...,a_N) = \lambda (b_0,...,b_N) $ for some $\lambda \in \C$. We denote by $[a_0: \cdots : a_N]$ the equivalence class of $(a_0,...,a_N)$. The mapping $\pi: \C^{N + 1} \setminus \{0\} \rightarrow \Pro^N(\C)$ given by $\pi(a_0,...,a_N) = [a_0: \cdots : a_N]$ is called the standard projection mapping.
For details see \cite[pp.99-102]{ru2001nevanlinna} and \cite{fujimoto1993value}.

A subset $H$ of $\Pro^N(\C)$ is called a hyperplane if there is an $N-$ dimensional linear subspace $\Tilde{H}$ of $\C^{N+1}$ such that $\pi(\Tilde{H} \setminus \{0\})= H$. For a fixed system of homogeneous coordinates $z = [z_0: z_1:\cdots :z_N]$, a hyperplane $H$ of $\PnC$ can be written as 
$$H = \{z \in \PnC | <z,\alpha> = 0\}$$
where $<z,\alpha> := a_0z_0 + \cdots + a_Nz_N$
and $\alpha = (a_0,\hdots,a_N) \in \C^{N+1}$ is a non-zero vector. 
We write it as
\begin{equation}\label{hyperplane eqn}
  H = \{<z,\alpha> = 0\}  
\end{equation}
for convenience. We also define 
$$||H|| := \max_{0 \leq i \leq N}|a_i|.$$
Henceforth in this paper, we only consider normalized hyperplane representations so that $||H|| = 1.$

If  $a_j(z), j = 0,1, \hdots ,N,$ are  holomorphic functions on a domain $D$ not all simultaneously zero at any point $z \in D$, then at each point $z \in D$ there is a hyperplane given by
$$H= \{z \in \PnC |  a_0(z)z_0 + \cdots + a_N(z)z_N  = 0\}$$ 
We call it a moving hyperplane,
and write it as
\begin{equation}\label{moving hyperplane eqn}
  H(z) = \{<z,\alpha(z)> = 0\}  
\end{equation}
Let $H_1,H_2, \hdots, H_{N+1}$ be hyperplanes in $\PnC$ represented as $H_j = \{<z,\alpha_j> = 0\}$ for $j = 1,\hdots,N+1.$
Then $$\D(H_1, \hdots, H_{N+1}) := |\det(\alpha_1^t,\hdots,\alpha_{N+1}^t)|$$
depends only on $H_j$ and not on any choice of $\alpha_j$ for normalized representations of hyperplanes.
\begin{definition}
    Let $H_1, \hdots, H_q$ with $q \geq N+1$, be hyperplanes in $\PnC$. Define 
    $$\D(H_1,\hdots, H_q) := \prod_{1\leq j_1 <\cdots<j_{N+1}\leq q} |\det (\alpha_{j_1}^t, \hdots, \alpha_{j_{N+1}}^t)|.$$
    We say the hyperplane family $H_1, \hdots, H_q, $ $q \geq {N+1}$, in  $\PnC$ is in general position if $\D(H_1,\hdots,H_q) > 0.$
\end{definition}
\subsection{Reduced representation}
Let $D$ be a domain in $\C$, $f: D \rightarrow \PnC$ be a holomorphic map and $U$ be an open set in $D$. Any holomorphic map $\Tilde{f}:U \rightarrow \C^{N+1} $ such that $\pi(\Tilde{f}(z)) \equiv f(z)$ in $U$ is called a \textit{representation} of $f$ on $U$, where $\pi$ is the standard projection map.
\begin{definition}
    For an open set $U$ of $D$ we call $\Tilde{f} = (f_0,\hdots,f_N)$ a \textit{reduced representation} of $f$ on $U$ if $\Tilde{f}$ is a representation of $f$ and  
 $f_0,\hdots,f_N$ are holomorphic functions on $U$ without common zeroes. 

\end{definition}
For any such representation we put $||\Tilde{f}(z)|| := \left(\ \sum_{i=0}^N |f_i(z)|^2 \right)^{\frac{1}{2}}$
\begin{remark}
    Every holomorphic map of $D$ into $\PnC$ has a reduced representation on $D$ \cite[section 5]{fujimoto1974meromorphic}.
\end{remark}

\begin{remark} \label{reduced representations remark}
    Let $\Tilde{f} = (f_0,\hdots,f_N)  $ be a reduced representation of $f$ on a domain D. Then for any arbitrary nowhere zero holomorphic function $h$ on $D$, $(hf_0, \hdots, hf_N)$  is also a reduced representation of $f$. Conversely, for any other reduced representation $(g_0,\hdots,g_N)$ of $f$, each $g_i$ can be written as $g_i = hf_i$ for some nowhere zero holomorphic function $h$ on $D$.
\end{remark}

\begin{definition}
    A function $f:D \rightarrow \PnC $ with reduced representation $\Tilde{f} = (f_0,\hdots,f_N)$ is said to intersect a hyperplane $H$ given by \eqref{hyperplane eqn} at a point $z_0 \in D$ if 
    $$<\Tilde{f},H> \  := a_0f_0 + \cdots + a_Nf_N$$ 
    has a zero at $z_0$ and we then  write $f(z_0) \in H.$ 

    If $<\Tilde{f},H> \equiv 0$ on $D$ then we say that $f(D) \subset H$ and if $<\Tilde{f}(z),H> \neq 0 
 \ \forall z \in D$ then $f$ is said to omit $H$ on $D.$
\end{definition}

\subsection{Normal family}
\begin{definition}
    A family $\scF$ of holomorphic maps from a domain $D$ into $\PnC$ is said to be normal in $D$ if any sequence in $\scF$ contains a subsequence which converges uniformly on compact subsets of $D$  to a holomorphic map of $D$ into $\PnC$; and $\scF$ is said to be \textit{normal at a point $a$} in $D$ if $\scF$ is normal on some neighbourhood of $a$ in $D$.
    \end{definition}

    Using the Fubini-Study metric on $\PnC$, we see that a sequence $\{f_n\}_{n=1}^{\infty}$ of holomorphic maps of $D$ into $\PnC$ converges uniformly on compact subsets of $D$ to a holomorphic map $f$ if and only if, for any $a \in D$ each $f_n$ has a reduced representation  $$\Tilde{f_n} = (f_{n0}, \hdots, f_{nN})$$ on some fixed neighbourhood $U$ of $a$ in $D$   such that 
    $\{f_{ni}\}_{n=1}^{\infty}$ converges uniformly on  $U$ to a holomorphic function 
    $f_i$ (for $i= 0,1,\hdots,N$) on $U$ with the property that 
    $$\Tilde{f} = (f_0, f_1, \hdots, f_N)$$ is a reduced representation of $f$ on $U$.
    \subsection{Extension of derivative}
    Now we proceed to the extension of the concept of derivative of meromorphic functions to that of functions from $D$ into $\PnC$. Note that any sort of spherical derivative using Fubini Study metric would have been inadequate to handle the problem of sharing where we require the derivative again to be a function into $\PnC$. 

    In a successful attempt at extension of Theorem \ref{Schwick theorem} a suitable notion of derivative for this setting was provided by Ye et al \cite{ye2015extension} as  follows:
    
    Let $f= [f_0:\cdots:f_N] $ be a holomorphic map of $D$ into $\PnC$, $\mu \in \{0,1,\hdots,N\} $ with $f_{\mu} \not\equiv 0$, and $d(z)$ be a holomorphic function in $D$ such that $f_{\mu}^2/d$ and $W(f_{\mu}, f_i)/d$ , $(i=0,1,\hdots, N; i \neq \mu)$ are holomorphic functions without common zeroes. Here, as usual, 
    $$W(f_{\mu}, f_i) = \begin{vmatrix}
        f_{\mu} & f_{i}\\
        f'_{\mu}&  f'_{i}
    \end{vmatrix} $$
    denotes the Wronskian of $f_{\mu}$ and $f_{i}$.

    \begin{definition}\cite{ye2015extension}
        The holomorphic map induced by the map $$(W(f_{\mu},f_0), \hdots, W(f_{\mu},f_{\mu - 1}), f_{\mu}^2, W(f_{\mu},f_{\mu + 1}), \hdots , W(f_{\mu},f_{N})): D \rightarrow \C^{N+1}$$ is called the $\mu$th \textit{derived holomorphic map } of $f$ and we write 
        $$ \nabla_{\mu}f = [W(f_{\mu},f_0)/d, \hdots, W(f_{\mu},f_{\mu - 1})/d, f_{\mu}^2/d, W(f_{\mu},f_{\mu + 1})/d, \hdots , W(f_{\mu},f_{N})/d] . $$
        For simplicity, we shall write $\nabla_{0}f$ as $\nabla f$.
    \end{definition}
  \begin{remark}
      The definition of $\nabla_{\mu}f$ does not depend on the choice of a reduced representation of $f$.
  \end{remark}  
\begin{remark}
    A meromorphic function $f= \displaystyle\frac{f_1(z)}{f_0(z)}$ on $D$  can be regarded as a holomorphic map from $D$ to $\Pro^1(\C)$ given by $f(z) = [f_0(z):f_1(z)]$ 
    and $\nabla f$ is exactly the ordinary derivative of $f(z)$.
\end{remark}
In 2015 Ye et al obtained the following result which led them to an extension of the Schwick's theorem.
\begin{theorema}[\cite{ye2015extension}]\label{YPY extension of Schwick's theorem}
     Let $\scF$ be a family of holomorphic maps of a domain $D \subseteq \C$ into $\PnC$, $H_1, \hdots, H_{2N+1}$ be hyperplanes in $\PnC$ in general position, and $\delta$ be a real number with $0 < \delta < 1$. Suppose that for each $f \in \scF$ the following conditions are satisfied: 
     \begin{enumerate}[(i)]
         \item If $\nabla f(z)  \in H_j$, then $f(z) \in H_j$, for $j=1,\hdots, 2N + 1$.
         \item If $f(z) \in \bigcup_{j=1}^{2N+1}H_j$, then 
         $$\frac{|<f(z),H_0>|}{||f(z)||. ||H_0||} \geq \delta $$
         where $H_0 = \{x_0 = 0\}$ is a coordinate hyperplane. 
         \item If $f(z) \in \bigcup_{j=1}^{2N+1}H_j$, then 
         $$\frac{|<\nabla f(z), H_k >|}{|f_0(z)|^2} \leq \frac{1}{\delta}$$
         for $k= 1, \hdots,2N+1$.
     \end{enumerate}
     Then $\scF$ is normal in $D$.
 \end{theorema}
 \subsection{Sharing of functions and hyperplanes}

 We now give the definition of sharing hyperplanes, which extends the definition of sharing values.
\begin{definition}
    Suppose $f$ and $g$ are two holomorphic functions from  domain $D$ into $\PnC$ and $H$ be a hyperplane in $\PnC$. If there exists some (and thus all) reduced representations $\Tilde{f}$ and $\Tilde{g}$ of $f$ and $g$ respectively such that $<\Tilde{f},H>$ and $<\Tilde{g},H>$ share(or partially share) zero on $D$, we say that $f$ and $g$ share(respectively partially share) $H$ on $D$.
\end{definition}
By remark \ref{reduced representations remark}, $<\Tilde{f}(z),H> = 0$ is independent of the choice of reduced representation of $f$. Therefore sharing(or partial sharing) of hyperplanes is well defined.

We can similarly define sharing(or partial sharing) of a moving hyperplane $H(z)$ by the functions $f$ and $g$ if  $<\Tilde{f},H(z)>$ and $<\Tilde{g},H(z)>$ share(respectively partially share) zero on $D$.
\section{Main results}

 In this paper we generalize the theorem \ref{YPY extension of Schwick's theorem} to the case where we can take moving  hyperplanes that can vary for different functions of the family  $\scF$. 
 \begin{theorem}\label{Generalization YPY}
     Let $\scF$ be a family of holomorphic maps of a domain $D \subseteq \C$ into $\PnC$  and $\delta$ be a real number with $0 < \delta < 1$. Suppose that for each $f \in \scF$ and $z \in D$ there exist $2N + 1 $ hyperplanes $H_1^f(z), \hdots, H_{2N+1}^f(z)$  in $\PnC$ in general position with 
     $$ \D(H_1^f(z), \hdots, H_{2N+1}^f(z))   > \delta , \ \forall f \in \scF, z \in D $$ such  that the following conditions are satisfied: 
     \begin{enumerate}[(i)]

         \item If $\nabla f(z)  \in H_j^f(z)$, then $f(z) \in H_j^f(z)$, for $j=1,\hdots, 2N + 1$.
         \item If $f(z) \in \bigcup_{j=1}^{2N+1}H_j^f(z)$, then 
         $$\frac{|<f(z),H_0>|}{||f(z)||. ||H_0||} \geq \delta $$
         where $H_0 = \{x_0 = 0\}$ is a coordinate hyperplane. 
         \item If $f(z) \in \bigcup_{j=1}^{2N+1}H_j^f(z)$, then 
         $$\frac{|<\nabla f(z), H_k^f(z) >|}{|f_0(z)|^2} \leq \frac{1}{\delta}$$
         for $k= 1, \hdots,2N+1$.
     \end{enumerate}
     Then $\scF$ is normal in $D$.
 \end{theorem}
In particular when $N=1$ each of the moving hyperplane $H_j^f(z), j=1,2 \ \text{and} \ 3$ is a function from the domain $D$ into $\Pro^1(\C)$ and we obtain the following result. 
\begin{theorem}\label{Similar to Grahl result 1}
    Let $\scF$ be a family of meromorphic functions on a domain $D \subseteq \C$ and $ \epsilon, M $ be two positive real numbers. Suppose for each $f \in \scF$ there exist three holomorphic functions $a_1^f, a_2^f \ \text{and} \ a_3^f$ defined on $D$ such that the following conditions are satisfied: 
\begin{enumerate}[(i)]
\item $ |a_j^f(z)| \leq M \ \text{and} \ |a_j^f(z) - a_k^f(z)| \geq \epsilon \ \forall z \in D \ \text{and} \ j,k \in \{1,2,3\} \ \text{with} \ j \neq k$
    \item For any $z \in D$ if $f'(z) = a_j^f(z)$ then $f(z) = a_j^f(z),$ for $j \in \{1,2,3\}$.
    \item For any $z \in D$ if $f(z) = a_j^f(z)$ then $|f'(z)| \leq M$ for $j \in \{1,2,3\}$.
\end{enumerate}
Then $\scF$ is normal in $D$.
\end{theorem}
The following examples show that each of the condition in theorem \ref{Similar to Grahl result 1} is essential.We denote the unit disk $\{z \in \C: |z| < 1\}$ by $\mathbb{D}$.
\begin{example}[Condition (i) is essential] 
    Consider the family  $\{f_n(z) = e^{nz}: n \in \mathbb{N}\}$ which is not normal in $\mathbb{D}.$ Take $a_j^{f_n}(z) = (j+1)ne^{nz}$ for $j \in \{1,2,3\}$. Note that $f_n'(z) = ne^{nz} \neq a_j^{f_n}(z)  $ for any $z \in \mathbb{D}$ and $j \in \{1,2,3\}.$ Also $f_n(z) \neq a_j^{f_n}(z)$ for any  $j \in \{1,2,3\}.$ 
\end{example}
\begin{example}[Condition (ii) is essential] 
    Consider the family
    $$\left\{ f_n(z) = \displaystyle\frac{1}{\sin nz} \left(\frac{1}{2} + \frac{z}{8n}\right) : n \in \mathbb{N}\right\}$$ of meromorphic functions  in  $\mathbb{D}$ which is not normal. Take $$a_1^{f_n}(z) = \displaystyle\frac{1}{2} + \displaystyle\frac{z}{8n}, a_2^{f_n}(z) = - \displaystyle\frac{1}{2} - \displaystyle\frac{z}{8n} \ \text{and} \ a_3^{f_n}(z) \equiv 0$$ which are holomorphic functions on $\mathbb{D}$ satisfying condition (i) of Theorem \ref{Similar to Grahl result 1} with $M = 1 \ \text{and} \ \epsilon = \frac{3}{8}$. We see that 
    $$f_n'(z) = \displaystyle\frac{-n \cos nz}{\sin^2 nz}\left(\displaystyle\frac{1}{2} + \displaystyle\frac{z}{8n}\right) + \frac{1}{\sin nz}\left(\frac{1}{8n}\right)$$ 
    Note that $f_n(z) = a_1^{f_n}(z)$ implies $\sin nz = 1$ and consequently $\cos nz = 0$. Then   $|f_n'(z)|$ is clearly bounded by 1 for all $z \in \mathbb{D}$ and $n \in \mathbb{N}$. Similarly we can see that $f_n(z) = a_2^{f_n}(z)$ implies $\sin nz = -1$ which leads to $\cos nz = 0$ and again $|f_n'(z)|$ is bounded by 1.
    Also it is clear that $f_n(z)$ omits $a_3^{f_n}(z)$ on $\mathbb{D}$. Thus condition (iii) of Theorem \ref{Similar to Grahl result 1} is also satisfied. 
    
\end{example}
\begin{example}[Condition (iii) is essential] 
    We again consider the family $\{f_n(z) = nz: n \in \mathbb{N}\}$ on $\mathbb{D}$. Taking $a_j^{f_n}(z) \equiv \displaystyle\frac{1}{j+1}, j \in \{1,2,3\}$ on $\mathbb{D}$ we can see that condition (i) of Theorem \ref{Similar to Grahl result 1} is satisfied with $\epsilon = \frac{1}{12}$ and $M = 1$ whereas condition (ii) is vacuously true.
\end{example}

In case $f(z)$ and $f'(z)$ share the functions $a^f_j(z)$ then it is easy to see that the condition $(ii)$ of Theorem \ref{Similar to Grahl result 1} is fulfilled and so we have the following Corollary which is a weaker form of Theorem \ref{Grahl and Nevo}.

\begin{corollary}\label{Similar to Grahl result 2}
  Let $\scF$ be a family of meromorphic functions on a domain $D \subseteq \C$. Suppose for each $f \in \scF$ there exist three holomorphic functions $a_1^f, a_2^f \ \text{and} \ a_3^f$ defined on $D$ and some  positive real numbers $\epsilon$ and $M$,with 
    $$ |a_j^f(z)| \leq M \ \text{and} \ |a_j^f(z) - a_k^f(z)| \geq \epsilon \ \forall z \in D \ \text{and} \ j,k \in \{1,2,3\} \ \text{with} \ j \neq k. $$ 
    such that $f$ and $f'$ share the functions $a_1^f, a_2^f \ \text{and} \ a_3^f$. 
Then $\scF$ is normal in $D$.  
\end{corollary}

Let $f(z)$ be a holomorphic function on a domain $D \subset \C$. Corresponding to this function we define a function $F: D \rightarrow \PtwoC$ whose representation 
$\Tilde{F}: D \rightarrow \C^3 \setminus \{0\}$ is given by 
\begin{equation}\label{F def}
    \Tilde{F}(z) = (1,z,f(z))
\end{equation}

Note that this representation is always a reduced representation with $\Tilde{F}_0 \neq 0$ everywhere on $D$. The derivative of $F$ (or more precisely the representation of derivative of $F$) shall be given by 
\begin{equation}\label{derivative defn}
   \nabla F = (1, 1, f'(z)) 
\end{equation}

which again happens to be in reduced representation for the whole domain $D$.

Taking $f$ to be a meromorphic function in representation \eqref{F def}  would have forced us to  use two different holomorphic functions in  \eqref{F def} and we would not have been able to obtain such a convenient form of derivative as in equation \eqref{derivative defn}.

We need the following lemma to relate the normality of families of complex valued holomorphic functions with the normality of  corresponding functions lifted to $\PtwoC$.

\begin{lemma}\label{Normality in new representation}
    Let $D$ be a domain in $\C$ and $\scF$ be a family of complex valued holomorphic functions on $D$. The family  $ \scF ^* =  \{F: D \rightarrow \PtwoC  \ \text{with} \  \Tilde{F}(z) = (1,z,f(z)) : f \in \scF \}$ is normal in $D$ if and only if $\scF$ is normal in $D$.
\end{lemma}
 Using our representation of complex valued holomorphic functions as holomorphic curves in $\PtwoC$, and a judicious choice of hyperplanes in $\PtwoC$ we obtain the following normality criterion which requires among other things a boundedness condition on derivatives at fixed points of the respective functions.  
 \begin{theorem}\label{Normality theorem for complex valued holomorphic functions}
     Let $\scF$ be a family of complex valued holomorphic functions on a domain $D \subseteq \C$. Let $A$ and $B$ be two disjoint compact subsets of unit disk $\mathbb{D}$ not containing the origin and $M$ be a positive number. Suppose for each $f \in \scF$ there exist two holomorphic functions $a_1^f(z)$ and $a_2^f(z)$ defined on $D$ and having range in $A$ and $B$ respectively such that the following conditions are satisfied :
     \begin{enumerate}[(i)]
         \item  $f'(z) = a_j^f(z)$ implies $f(z) = a_j^f(z) $ and $f(z)= a_j^f(z)$ implies $|f'(z)| \leq M$ for $j \in \{1,2\}$. 
         \item $z$ is a fixed point of $f$ wherever $f'$ assumes 1 in $D$.
         \item $|f'(z)| \leq M$, whenever $z$ is a fixed point of $f$.
         
     \end{enumerate}
     Then $\scF$ is a normal family.
 \end{theorem}
That condition $(i)$ is essential  can be seen by the following example:
\begin{example}
 Let $\scF = \{f_n(z) = e^{nz} + z: n \in \mathbb{N}\}$ defined on $\C$. Here each $f_n'(z) = ne^{nz} + 1 \neq 1$ for any $z \in \C$ . Also it is clear that no member of this family can have any fixed point. But $\scF$ is not a normal family. 
\end{example}
Theorem \ref{Normality theorem for complex valued holomorphic functions} enables us to formulate a counterexample to \textbf{the converse of Bloch's principle}\cite{bergweiler2006bloch}: \emph{A family of holomorphic functions which have property P in common in a domain $D$ is (apt to be) a normal family in $D$ if P can not be possessed by non-constant entire functions.}
\begin{counterexample}
    Consider the function $f(z) = 2z \ \forall z \in \C$. Then taking $a_1^f(z) \equiv \frac{1}{2}, a_2^f(z) \equiv \frac{-1}{2}$ and $M = 2$ we see that $f$ satisfies all the conditions required in Theorem \ref{Normality theorem for complex valued holomorphic functions} and is a non-constant entire function.
\end{counterexample}
\section{Proofs}

To prove our main results we shall require the following generalized version of Zalcman lemma
\begin{lemma}[\cite{aladro1991criterion}]\label{Aladro Krantz}
    Let $\scF$ be a family of holomorphic maps of a domain $D$ in $\C$ into $\PnC$. The family $\scF$ is not normal on $D$ if and only if there exist sequences $\{f_n\} \subset \scF, \{z_n\}  \subset D$ with $z_n \rightarrow z_0 \in D, \{\rho_n\}$ with $ \rho_n  > 0 $ and $\rho_n \rightarrow 0$ such that 
    $$g_n(\xi) := f_n(z_n + \rho_n \xi )$$
    converges uniformly on compact subsets of $\C$ to a nonconstant holomorphic map $g$ of $\C$ into $\PnC$.
\end{lemma}
Also from the degenerate second main theorem in Nevanlinna theory we have the following fact: 
\begin{lemma}[see \cite{ru2001nevanlinna}, p.141]\label{picard extension}
 Let $f: \C \rightarrow \PnC$ be a holomorphic map, and $H_1,\hdots, H_q(q \geq 2N+1)$ be hyperplanes in $\PnC$ in general position. If for each $j= 1,\hdots ,q$, either $f(\C) $ is contained in $H_j$, or $f(\C)$ omits $H_j$, then $f$ is constant.   
\end{lemma}
\begin{proof}[Proof of Theorem \ref{Generalization YPY}]
    Suppose $\scF$ is not normal on $D$. Then by lemma \ref{Aladro Krantz} there exist sequences $\{f_n\} \subset \scF, \{z_n\}  \subset D$ with $z_n \rightarrow z_0 \in D, \{\rho_n\}$ with $ \rho_n  > 0 $ and $\rho_n \rightarrow 0$ such that 
    $$g_n(\xi) := f_n(z_n + \rho_n \xi )$$
    converges uniformly on compact subsets of $\C$ to a nonconstant holomorphic map $g : \C \rightarrow \PnC$. Corresponding to each such $f_n$ let  $H_j^{f_n}$, $j= 1,2, \hdots, 2N+1$, be the moving hyperplanes given by the hypothesis of the theorem. Let each $H_j^{f_n}$ be written in normalized representation as 
    $$H_j^{f_n} = \{<z,\alpha_j^{f_n}(z)> = 0 \}$$
    where $\alpha_j^{f_n}(z) = (\alpha_{j,0}^{f_n}(z), \alpha_{j,1}^{f_n}(z), \hdots, \alpha_{j,N}^{f_n}(z))$ is a non-zero  vector in $\C^{N+1}$ for any $z \in D$ with $$\max_{k=0}^{N}|\alpha_{j,k}^{f_n}(z)| = 1$$ 
    By compactness of $A = \{(\alpha_0, \hdots , \alpha_n) \in \C^{N+1} : \displaystyle\max_{k=0}^{N}|\alpha_k| = 1\}$ it follows that for any $j \in \{1,2,\hdots, 2N+1\}$ there exist a subsequence of $\alpha_j^{f_n}(z)$ which we again call $\alpha_j^{f_n}(z)$ that converges to $\alpha_j(z)$ locally uniformly on $D$ with $\alpha_j(z) \in A$ for any $z \in D$. Let $H_j(z)$ be the moving hyperplane defined by $\alpha_j(z) , j = 1,2, \hdots , 2N+1$. By  continuity of the determinant function it is clear that for any $z \in D$ 
    $$\D(H_1(z),\hdots, H_{2N+1}(z) ) \geq \inf \{\D (H_1^{f_n}(z),\hdots, H_{2N+1}^{f_n}(z)): n \in \mathbb{N}\} \geq \delta > 0$$ and so the hyperplanes $H_1(z), \hdots, H_{2N+1}(z)$ are in general position for any $z \in D$ and in particular at $z_0$. As $g: \C \rightarrow \PnC$ is a non constant  function, by lemma \ref{picard extension} we may assume that $<g(\xi),H_1(z_0)>$ does not vanish identically and has at least one zero in $\C$. Let $\xi_0$ be such a zero. That is
    \begin{equation}\label{g in first hyperplane}
     <g(\xi_0), H_1(z_0)> = 0   
    \end{equation}
    
    and there exists a small neighbourhood $U$ of $\xi_0$ such that $<g(\xi),H_1(z_0)>$  has no other zero in $U$. Moreover, each $g_n$ has a reduced representation
    $$\Tilde{g_n}(\xi) = (g_{n,0}(\xi),\hdots, g_{n,N}(\xi)) = (f_{n,0}(z_n + \rho_n \xi), \hdots, f_{n,N}(z_n + \rho_n \xi) )$$
    on $U$ such that for each $i = 0,1, \hdots , N$ the sequence $\{g_{n,i}: n \in \mathbb{N}\}$ converges uniformly on $U$ to a holomorphic function $g_i$ with $$\Tilde{g} = (g_0, \hdots, g_N)$$ as the reduced representation of $g$ on $U$. Shrinking the neighborhood $U$ if necessary we can see that $\displaystyle\sum_{i=0}^N\alpha_{1,i}^{f_n}(z_n + \rho_n\xi)g_{n,i}(\xi)$ converges uniformly on $U$ to $\displaystyle\sum_{i=0}^Na_{1,i}(z_0)g_{i}(\xi)$. 
Note that from equation \eqref{g in first hyperplane}, 
$$\sum_{i=0}^N \alpha_{1,i}(z_0)g_i(\xi_0) = 0$$
    So by Hurwitz's theorem, there exists a  sequence $\xi_n$ tending to $\xi_0$ such that, for large $n$, $\displaystyle\sum_{i=0}^N\alpha_{1,i}^{f_n}(z_n + \rho_n\xi_n)g_{n,i}(\xi_n) = 0$, that is 
    \begin{equation}\label{f in hyperplane}
       \displaystyle\sum_{i=0}^N\alpha_{1,i}^{f_n}(z_n + \rho_n\xi_n)f_{n,i}(z_n + \rho_n\xi_n) = 0 
    \end{equation}
  which implies   $ f_{n}(z_n + \rho_n\xi_n) \in H_1^{f_n}(z_n + \rho_n\xi_n)$ and by condition $(ii)$ of the hypothesis we shall have 
  $$|f_{n,0}(z_n + \rho_n \xi_n)| \geq \delta ||\Tilde{f_n}(z_n + \rho_n \xi_n)||$$
  That is 
  $$|g_{n,0}( \xi_n)| \geq \delta ||\Tilde{g_n}(\xi_n)||$$
    Taking $n \rightarrow \infty$ we get 
    $$|g_{0}( \xi_0)| \geq \delta ||\Tilde{g}(\xi_0)|| > 0$$
    Thus, $g_0(\xi_0) \neq 0$ and further shrinking $U$ if necessary we may assume that $g_0(\xi) \neq 0 \ \forall \ \xi \in U$. This shall imply $g_{n,0}(\xi) \neq 0$ for $ \xi \in U$ when $n$ is sufficiently large. 
    For each such $n$ 
    $$(g_{n,0}^2(\xi), W(g_{n,0}, g_{n,1})(\xi), \hdots,W(g_{n,0}, g_{n,N})(\xi) )$$ 
    is then a reduced representation of $\nabla g_n(\xi)$ on $U$. Also note that for any $i =1,2, \hdots, N$ and $\xi \in U$ the Wronskian 
    \begin{equation}\label{wronskian f and g}
       W(g_{n,0}, g_{n,i})(\xi) = \rho_n W(f_{n,0}, f_{n,i})(z_n + \rho_n \xi)
    \end{equation}
   
    From \eqref{f in hyperplane} and condition $(iii)$ of the hypothesis we shall have for any $k= 2,3,\hdots, 2N + 1$
    
    $$\frac{|<\left(f_{n,0}^2,W(f_{n,0},f_{n,1}) ,\hdots, W(f_{n,0},f_{n,N}) \right), (\alpha^{f_n}_{k,0},\hdots, \alpha^{f_n}_{k,N} ) >|}{|f_{n,0}^2|} \leq \frac{1}{\delta}$$
    where the LHS of above inequality is evaluated at $(z_n + \rho_n \xi_n)$
    $$\implies \left|\alpha^{f_n}_{k,0}(z_n + \rho_n \xi_n) + \sum_{i=1}^{N}\frac{\alpha^{f_n}_{k,i}(z_n + \rho_n \xi_n)W(f_{n,0},f_{n,i})(z_n + \rho_n \xi_n)}{f_{n,0}^2(z_n + \rho_n \xi_n)}  \right | \leq \frac{1}{\delta}$$
Using \eqref{wronskian f and g} we get
$$\left|\alpha^{f_n}_{k,0}(z_n + \rho_n \xi_n) + \sum_{i=1}^{N}\frac{\alpha^{f_n}_{k,i}(z_n + \rho_n \xi_n)W(g_{n,0},g_{n,i})(\xi_n)}{\rho_n g_{n,0}^2(\xi_n)}  \right | \leq \frac{1}{\delta}$$
\begin{equation}\label{sum of wronskians}
\implies \left|\rho_n\alpha^{f_n}_{k,0}(z_n + \rho_n \xi_n) + \sum_{i=1}^{N}\frac{\alpha^{f_n}_{k,i}(z_n + \rho_n \xi_n)W(g_{n,0},g_{n,i})(\xi_n)}{ g_{n,0}^2(\xi_n)}  \right | \leq \frac{\rho_n}{\delta}
\end{equation}
Define for any $k= 2,3, \hdots, 2N +1$ $$\phi_k(\xi) = \sum_{i=1}^{N}\frac{\alpha_{k,i}(z_0)W(g_{0},g_{i})(\xi)}{ g_{0}^2(\xi)} , \xi \in U  $$
Taking $n \rightarrow \infty $ in equation\eqref{sum of wronskians}, we obtain
\begin{equation}\label{phi is zero at xi}
    \phi_k(\xi_0) = 0 \ \text{for} \ k= 2,3, \hdots, 2N +1
\end{equation}

Now we claim that there are at most $N$ hyperplanes in $\{H_k: k= 2,3, \hdots, 2N +1\}$ such that $\phi_k(\xi) \equiv 0$ on $U$.
Suppose on the contrary that $\phi_2(\xi) \equiv \cdots \phi_{N+2}(\xi) \equiv 0$ on $U$. Since $$\phi_j(\xi) = \left( \sum_{i=1}^{N}\alpha_{j,i}(z_0)\frac{g_i}{g_0}\right)' \equiv 0,$$
there exist complex numbers $c_j$ such that 
$$\sum_{i=1}^{N}\alpha_{j,i}(z_0)\frac{g_i}{g_0} \equiv c_j \ \ \  \text{on} \ U \ \text{for} \ j= 2, \hdots N+2  $$

Since $H_1(z_0), \hdots, H_{2N+1}(z_0)$ are in general position, the system of equations 
$$\sum_{i=1}^{N}\alpha_{j,i}(z_0)x_i \equiv c_j \ \text{for} \ j= 2, \hdots , N+2 $$
 has no solution or the solution is unique. Thus $\Tilde{g}$ is uniquely determined by $g_0 \neq 0$ in $U$. Hence $g$ is constant on $U$ which is a contradiction. So the claim is true.
 By the validity of our claim we can suppose without any loss of generality that $\phi_k(\xi) \not\equiv 0 $ on $U$ for $k= 2, \hdots, N+1$. Also for each $k \in \{2, \hdots, N+1\}$ 
 $$\rho_n\alpha^{f_n}_{k,0}(z_n + \rho_n \xi) + \rho_n\sum_{i=1}^{N}\frac{\alpha^{f_n}_{k,i}(z_n + \rho_n \xi)W(f_{n,0},f_{n,i})(z_n + \rho_n\xi)}{ f_{n,0}^2(z_n + \rho_n\xi)} 
$$
converges uniformly to $\phi_k(\xi)$ on $U$. Again by Hurwitz theorem and equation \eqref{phi is zero at xi} there exists a sequence $\zeta_n \rightarrow \xi_0$ such that 
$$<\nabla f_n(z_n + \rho_n\zeta_n), H^{f_n}_k(z_n + \rho_n\zeta_n)> = 0$$
    Then by condition $(i)$ of the hypothesis 
    $$< f_n(z_n + \rho_n\zeta_n), H^{f_n}_k(z_n + \rho_n\zeta_n)> = 0$$

Taking $n \rightarrow \infty$ we get
$$<g(\xi_0), H_k(z_0)> = 0 \ \text{for} \ k = 2, \hdots, N+1 $$
Noting equation \eqref{g in first hyperplane} we shall have 
$$<g(\xi_0), H_k(z_0)> = 0 \ \text{for} \ k = 1,2, \hdots, N+1$$
This is a contradiction to the fact that the hyperplanes $\{H_k(z_0): k = 1, 2, \hdots, N+1\}$ are in general position. Hence $\scF$ is normal on D.

\end{proof}
\begin{proof}[Proof of Theorem \ref{Similar to Grahl result 1}]
    We can assume that $M = 1$ for otherwise we may replace the functions $a_j^f(z)$ by $\displaystyle\frac{a_j^f(z)}{M}, j \in \{1,2,3\}$ and consider normality of the family $\scF^* = \{\frac{f}{M}: f \in \scF\}$ which is equivalent to the normality of $\scF.$ 
    
    Now for any $f \in \scF, z \in D, j \in \{1,2,3\}$; let $H_j^f(z)$ be the hyperplane defined by $<z, \alpha_j^f(z)> = 0,$ where $\alpha_j^f(z) = (a_j^f(z), -1) \in \C^2$ gives a normalized representation for $H_j^f(z)$. Writing the meromorphic function $f(z) = \displaystyle\frac{f_1(z)}{f_0(z)}$ as $f(z) = [f_0(z),f_1(z)]$, it can be seen that for any $z_0 \in D$, $f(z_0) = a_j^f(z_0)$ if and only if $f(z_0) \in H_j^f(z_0),$ and  $f'(z_0) = a_j^f(z_0)$ if and only if $\nabla f(z_0) \in H_j^f(z_0).$ 
    It can easily be computed that
    \begin{equation}\label{Hyperplanes independent in GR1}
     \D(H_1^f(z), H_2^f(z), H_3^f(z)) \geq \epsilon^3 \ \forall f \in \scF \ \text{and} \ z \in D   
    \end{equation}

    Further $f(z) \in \bigcup_{j=1}^3 H_j^f(z)$ implies $f(z) = a_j^f(z)$ for some $j \in \{1,2,3\}$ and we shall have
    \begin{equation}\label{Second condition of Theorem true in GR1}
    \displaystyle\frac{|f_0(z)|}{||f(z)||} = \displaystyle\frac{1}{\sqrt{1 + \left|\frac{f_1(z)}{f_0(z)}\right|^2}} \geq \displaystyle\frac{1}{\sqrt{2}}    
    \end{equation}
and 
\begin{equation}\label{Third condition of theorem true in GR1}
    \frac{|<\nabla f(z), H_k^f(z)>|}{|f_0(z)|^2} = |a_k^f(z) - f'(z)| \leq 2
\end{equation}
    Taking $\delta = \min\{\epsilon^3, \frac{1}{2}\}$, we see from equations \eqref{Hyperplanes independent in GR1}, \eqref{Second condition of Theorem true in GR1} and \eqref{Third condition of theorem true in GR1} that $\scF$ satisfies all the conditions of the Theorem \ref{Generalization YPY}. Hence $\scF$ is normal in $D.$
\end{proof}
\begin{proof}[Proof of Lemma \ref{Normality in new representation}]
  To prove this lemma it is enough to show that a sequence $\{f_n\} \subset \scF$ converges locally uniformly on $D$ to a holomorphic function $f$ on $D$ or converges locally uniformly to $\infty$  if and only if the sequence $\{F_n\}$ with $\Tilde{F_n} = (1,z,f_n(z))$ converges locally uniformly on  $D$ to a holomorphic function $F: D \rightarrow \PtwoC$. 
  
   First suppose that $\{f_n\} \subset \scF$ converges locally uniformly on $D$ to a holomorphic function $f$ then clearly $F_n$ converges locally uniformly to $F$ whose reduced representation is $\Tilde{F} = (1,z,f(z))$. On the other hand if $f_n $ converges locally uniformly to $\infty$ then we note that for any $z_0 \in D \ \exists $ a neighborhood $U$ such that for sufficiently large $n, f_n(z) \neq 0 \ \forall \ z \in U$. On such a neighborhood $U$ we can see that $$\left(\frac{1}{f_n(z)},\frac{z}{f_n(z)},1\right)$$ is again a reduced representation of $F_n$ which converges uniformly on $U$ to the function $F: U \rightarrow \PtwoC$ with reduced representation $\Tilde{F}(z) \equiv  \left(0,0,1\right)$.

Conversely, suppose $F_n \rightarrow F$.Let $\zeta_0 \in D$. Then there exists a neighborhood $U$ of $\zeta_0$ and reduced representation $(\alpha_n(z), \beta_n(z), \gamma_n(z))$ of $F_n(z)$ in $U$ such that $\alpha_n \rightarrow \alpha, \beta_n \rightarrow \beta$ and $\gamma_n \rightarrow \gamma$  uniformly on $U$ where $(\alpha(z), \beta(z), \gamma(z))$ is a reduced representation of $F(z)$ in $U$. As $\Tilde{F_n}(z) = (1,z,f_n(z))$ is also a reduced representation of $F_n(z)$ there  are no where zero holomorphic functions $h_n(z)$ on $U$ such that for all $n \in \mathbb{N}$
\begin{align}\label{two reduced representations 1}
    \alpha_n(z) &= 1.h_n(z)\\
    \label{two reduced representations 2}
    \beta_n(z) &= z.h_n(z)\\
    \label{two reduced representations 3}
    \gamma_n(z) &= f_n(z)h_n(z)
\end{align}

\textbf{Claim}:
$\alpha(z)$ is either identically zero or nowhere zero on $U$. 

Suppose on the contrary that $\alpha(z) \not\equiv 0$ and $\alpha(z_0) = 0$ for some $z_0 \in U$. Then by Hurwitz theorem there exist a sequence $\{z_n\} \subset U$ converging to $z_0$ such that for sufficiently large $n, \alpha_n(z_n) = 0$. Using this in equation \eqref{two reduced representations 1}  we obtain $h_n(z_n) = 0$ which is a contradiction. This proves our claim. 

Now we consider both the possible cases:

\textbf{Case 1}: $\alpha(z) \neq 0 $ for all $z \in U$

Then by uniform convergence of $\alpha_n(z)$ to $\alpha(z)$ in $U$, there exists $n_0 \in \mathbb{N}$ such that for $n \geq n_0$ 
$$\alpha_n(z) \neq 0 \ \forall z \in U$$

Using equation \eqref{two reduced representations 3} we can write for $z \in U \ \text{and} \ n \geq n_0$
$$f_n(z) = \frac{\gamma_n(z)}{\alpha_n(z)}\frac{\alpha_n(z)}{h_n(z)} \rightarrow \frac{\gamma(z)}{\alpha(z)} .1$$
Thus $f_n(z)$ converges  uniformly  to the holomorphic function $\displaystyle\frac{\gamma(z)}{\alpha(z)}$ on $U$.

\textbf{Case 2}: $\alpha(z) \equiv 0 $ on $ U$

By equation \eqref{two reduced representations 1} we shall have $h_n(z) \rightarrow 0$  uniformly on $U$.
With equation \eqref{two reduced representations 2} this implies that $\beta_n(z) \rightarrow 0$  uniformly on $U$. That is $\beta(z) \equiv 0$ on $U$. Therefore, $\gamma(z) \neq 0 \ \forall z \in U$ so that $(\alpha(z), \beta(z), \gamma(z))$ is a reduced representation of $F(z)$ on $U$.
As $$f_n(z)h_n(z)= \gamma_n(z) \rightarrow \gamma(z) \ \text{and} \ h_n(z) \rightarrow 0$$  uniformly on $U$, it follows that $f_n \rightarrow \infty$  uniformly on $U$.

\end{proof}

\begin{proof}[Proof of Theorem \ref{Normality theorem for complex valued holomorphic functions}]
  Since normality is a local property we take $D = \mathbb{D}$. We shall show that the family $\scF^*$ as defined in lemma \ref{Normality in new representation} is normal in $\mathbb{D}$. For any $f\in \scF$ we make the following choice of $5$ hyperplanes 
  $$H^f_j(z) = \{w \in \PtwoC| <w,\alpha^f_j(z)> = 0\},  j =1, \hdots, 5 $$
  where 
  \begin{align*}
  \alpha^f_1(z) &= (a_1^f(z),0, -1)\\
  \alpha^f_2(z) &= (a_2^f(z),0, -1)\\
  \alpha^f_3(z) &= (0, 1, -1)\\
  \alpha^f_4(z) &= \left(1, \frac{-1}{2}, 0\right)\\
  \alpha^f_5(z) &= \left(1,\displaystyle\frac{-1}{3}, 0\right)  
  \end{align*}
  Let $\eta_1 = \min\{|\zeta|: \zeta \in A\}, \delta_1 = \max\{|\zeta|: \zeta \in A\}$, 
  
  $\eta_2 = \min\{|\zeta|: \zeta \in B\}, \delta_2 = \max\{|\zeta|: \zeta \in B\}$ 
  
  and $\delta_3 = d(A,B) = \inf\{|\zeta - \xi|: \zeta \in A, \xi \in B\}.$
  
  Clearly $0< \eta_1,\delta_1, \eta_2,  \delta_2 <1$ and $0< \delta_3 < 2.$
  
We can compute that for any $f \in \scF$ and $z \in \mathbb{D}$
$$\D\left(H^f_1(z),H^f_2(z), H^f_3(z), H^f_4(z), H^f_5(z)\right)$$
$$= \frac{1}{6^4}\left|\frac{a_1^f(z)}{2} -1\right|\left|\frac{a_2^f(z)}{2} -1\right|\left|\frac{a_1^f(z)}{3} -1\right|\left|\frac{a_2^f(z)}{3} -1\right|\left|a_2^f(z)- a_1^f(z)\right|^3$$
$$\geq \frac{1}{6^4}
\left|1- \frac{\delta_1}{2}\right|
\left|1- \frac{\delta_2}{2}\right|
\left|1- \frac{\delta_1}{3}\right|
\left|1- \frac{\delta_2}{3}\right|
\left|\delta_3\right|^3 = \delta^* \text{(say)} >0 $$
  If we take $\delta = \min \{\delta^*, \displaystyle\frac{1}{\sqrt{3}}, \displaystyle\frac{1}{1+M}\},$ then it can be seen that $\scF^*$ satisfies all the conditions of Theorem \ref{Generalization YPY}. Hence $\scF^*$ is normal and so by lemma \ref{Normality in new representation} $\scF$ is also normal. 
\end{proof}
 
\end{document}